\documentclass[12pt]{article}
\usepackage{amsmath}
\usepackage{amsfonts}
\usepackage{amssymb}
\usepackage{amscd}
\usepackage{amsthm}
\usepackage{bbm}
\usepackage{color}      
\usepackage[linktocpage=true,plainpages=false,pdfpagelabels=false]{hyperref}

\input xypic

\usepackage[matrix,arrow,curve]{xy}

\newcommand{\Z}{{\mathbb Z}}

\newcommand{\Ker}{{\rm Ker}}

\newcommand{\rk}{{\rm rk}}

\newcommand{\Spec}{{\rm Spec} }

\newcommand{\gm}{\mathbb G_{m}}

\newcommand{\z}{{\mathbb{Z}}}

\newcommand{\lrto}{\longrightarrow}

\newcommand{\uz}{\underline{\mathbb{Z}}}

\theoremstyle{plain}
\newtheorem{theor}{Theorem}[section]
\newtheorem{prop}[theor]{Proposition}

\newtheorem{lemma}[theor]{Lemma}

\theoremstyle{remark}
\newtheorem{rmk}[theor]{Remark}
\newtheorem{examp}[theor]{Example}

\theoremstyle{definition}
\newtheorem{defin}[theor]{Definition}
\newtheorem{defin-prop}[theor]{Definition-Proposition}

\newcommand{\quash}[1]{}  

\title{Tangent space to Milnor $K$-groups of rings\footnotetext{This work is supported by the RSF under a grant 14-50-00005.}}
\author{S.\,O. Gorchinskiy, D.\,V. Osipov}
\date{}

\begin{document}
\maketitle


\begin{abstract}
We prove that the tangent space to the $(n+1)$-th Milnor $K$-group of a ring $R$ is isomorphic to group of $n$-th absolute K\"ahler differentials of $R$ when the ring $R$ contains $\frac{1}{2}$ and has sufficiently many invertible elements. More precisely, the latter condition is that $R$ is weakly $5$-fold stable in the sense of Morrow.
\end{abstract}

\section{Introduction}

The Milnor $K$-group $K^M_n(R)$ of a commutative associative unital ring $R$ is generated by symbols $\{r_1,\ldots,r_n\}$, $r_i\in R^*$, that satisfy the Steinberg relations (see Definition~\ref{defin:Milnor}). Studying Milnor $K$-groups, one often requires that $R$ has sufficiently many invertible elements. In this context, van der Kallen~\cite{vdK2} has introduced the notion of a $k$-fold stable ring for a natural number $k$ (see Remark~\ref{rmk:notstable}(i)). Recently Morrow~\cite{Mor} defined weakly $k$-fold stable rings (see Definition~\ref{defin:wstable}). Note that for any commutative associative unital ring $A$, the ring of Laurent series~$A((t))$ is weakly $k$-fold stable for all $k$, while for many natural rings $A$, the ring~$A((t))$ is not $k$-fold stable for any $k$ (see Remark~\ref{rmk:notstable}(ii)).

Let $\varepsilon$ be a formal variable such that $\varepsilon^2=0$. By a tangent space $TK^M_n(R)$ to Milnor $K$-group, we mean the kernel of the natural homomorphism $K^M_n\big(R[\varepsilon]\big)\to K^M_n(R)$ (see Definition~\ref{defin:tang}). Let $\Omega_R^n$ denote the group of $n$-th absolute K\"ahler differentials of $R$. Following Bloch~\cite{Blo}, one constructs a natural homomorphism (see Definition~\ref{defin:B})
$$
B\;:\;TK^M_{n+1}(R)\lrto \Omega^{n}_R\,.
$$
In particular, for any collection of invertible elements $r_1,\ldots,r_n\in R^*$ and any element ${s\in R}$, the homomorphism $B$ sends the symbol $\{1+sr_1\ldots r_n\,\varepsilon,r_1,\ldots,r_n\}$ to the differential form $sdr_1\wedge\ldots\wedge dr_n$ (see Example~\ref{rem:explB}).

The aim of the paper is to prove that $B$ is an isomorphism when $R$ contains $\frac{1}{2}$ and is weakly $5$-fold stable (see Theorem~\ref{thm:tangentMilnor}).

\medskip

Besides being of independent interest, this statement is also of utmost importance for the proof of the explicit formula and the universal property of the higher-dimensional Contou-Carr\`ere symbol. The proof will be given in the work~\cite{GO2} (the explicit formula was also announced in the short note~\cite{GO1}). In this proof, Theorem~\ref{thm:tangentMilnor} is applied to the ring of iterated Laurent series $R=A((t_1))\ldots((t_n))$ over a ring $A$. Thus it is important that Theorem~\ref{thm:tangentMilnor} is valid for any weakly $5$-fold stable ring and not only for a $5$-fold stable ring.

\medskip

Let us explain how to deduce Theorem~\ref{thm:tangentMilnor} in a weaker form from previously known results of van der Kallen and Bloch. Namely, in~\cite{vdK1} it was constructed a natural isomorphism $TK_2(R)\simeq \Omega^1_R$ for any ring $R$ that contains~$\frac{1}{2}$, where $TK_2(R)$ is the tangent space to the algebraic \mbox{$K$-group}~$K_2(R)$. Later it was proved in~\cite[Theor.\,7.1,\,8.4]{vdK2} that for a $5$-fold stable ring~$R$, there is an isomorphism $K_2^M(R)\simeq K_2(R)$, whence $TK_2^M(R)\simeq TK_2(R)$. The resulting isomorphism ${TK_2^M(R)\simeq \Omega^1_R}$ coincides with~$B$. This proves Theorem~\ref{thm:tangentMilnor} when~$n=1$ and $R$ is a \mbox{$5$-fold} stable ring that contains~$\frac{1}{2}$. Using the case $n=1$ and closely following the strategy from~\cite{Blo}, it is possible to obtain Theorem~\ref{thm:tangentMilnor} when $n$ is an arbitrary natural number and $R$ is a $5$-fold stable ring that contains~$\frac{1}{2}$ (but not a weakly $5$-fold stable ring as in Theorem~\ref{thm:tangentMilnor} itself). See also a recent paper of Dribus~\cite{Drib} for a more general statement about $5$-fold stable rings, which is an analog for Milnor $K$-groups of the famous theorem of Goodwillie~\cite{Goo}.

\medskip

Notice that the above approach to a weaker form of Theorem~\ref{thm:tangentMilnor} is based on the machinery of algebraic $K$-theory, in particular, on hard results of van~der~Kallen~\cite{vdK2} about elements in the Steinberg extension. However, the theorem is essentially a statement about relations between symbols in Milnor $K$-groups~$K_{n+1}^M\big(R[\varepsilon]\big)$ and it is natural to expect that there exists a direct proof in terms of symbols only. In this paper, we give such proof.

\medskip

We prove Theorem~\ref{thm:tangentMilnor} following the strategy from~\cite{Blo} but replacing van der Kallen's results on the algebraic $K$-group $K_2$ by explicit calculations with symbols in the Milnor $K$-group~$K^M_2\big(R[\varepsilon]\big)$ (see Lemmas~\ref{lemma:epseps} and~\ref{lemma:cool}). We believe that these calculations have their own interest as well.

\medskip

The authors are grateful to C.\,Shramov for helpful discussions and to the referee for suggestions concerning the exposition.

\section{Statement of the main result}\label{section:stat}

For short, by a ring we mean a commutative associative unital ring and by a {\it group functor}, we mean a covariant functor from the category of commutative associative unital rings to the category of abelian groups. Fix a natural number $n\geqslant 0$.

\subsection{Milnor $K$-groups}\label{Milnor-K-group}

Recall the following well-known definition.

\begin{defin}\label{defin:Milnor}
The {\it $n$-th Milnor $K$-group} $K_n^M(R)$ of a ring $R$ is the $n$-th homogenous component of the graded ring
$$
\mbox{$\bigoplus\limits_{n\geqslant 0}K_n^M(R):=\bigoplus\limits_{n\geqslant 0}(R^*)^{\otimes\,n}/ \, {\rm St}$}\,,
$$
where $\rm St$ is the homogenous ideal generated by all elements of type $r\otimes (1-r)\in R^*\otimes_{\z} R^*$.
\end{defin}

We denote the group law in $K_n^M(R)$ additively. Explicitly, $K_n^M(R)$ is the quotient group of the group~$(R^*)^{\otimes\,n}$ by the subgroup generated by all elements of type
$$
r_1\otimes\ldots\otimes r_i\otimes r\otimes(1-r)\otimes r_{i+3}\otimes\ldots\otimes r_n\,,
$$
which are called {\it Steinberg relations}. Note that in this tensor, $r$ and $1-r$ come one after another and are not separated. The class in $K^M_n(R)$ of a tensor \mbox{$r_1\otimes\ldots\otimes r_n$}, where $r_i\in R^*$, is denoted by~$\{r_1,\ldots,r_n\}$ and is usually called a {\it symbol}. Thus we do not require additional relations on symbols besides the multilinearity and the Steinberg relations. In particular, we have that $K^M_0(R)=\Z$ and $K_1^M(R)=R^*$.

\medskip

Clearly, $K_n^M$ is a group functor. Denote by $\Omega^n$ the group functor that sends a ring~$R$ to the group of $n$-th absolute K\"ahler differentials $\Omega^n_R$. It is easy to check that there is a morphism of group functors
$$
d\log\;:\; K_n^M\lrto \Omega^n\,,\qquad \{r_1,\ldots,r_n\}\longmapsto \frac{dr_1}{r_1}\wedge\ldots\wedge\frac{dr_n}{r_n}\,.
$$

\quash{
\medskip

For any ring $R$, we also have algebraic $K$-groups $K_m(R)$, which are functorial with respect to $R$ as well. There are canonical decompositions of group functors
$$
K_0\simeq \uz\times\widetilde{K}_0\,,\qquad K_1\simeq \gm\times SK_1\,.
$$
Denote by $\rk\colon K_0\to\uz$ and $\det\colon K_1\to \gm$ the corresponding projections. Explicitly, they are defined by taking the rank of a finitely generated projective modules and by taking the determinant of a matrix, respectively.

\medskip

For any ring $R$, Loday~\cite{L} constructed a graded-commutative product between algebraic $K$-groups, which is functorial with respect to $R$:
$$
K_i(R)\otimes_{\mathbb{Z}} K_j(R)\lrto K_{i+j}(R)\,,\qquad  a\otimes b \longmapsto a\cdot b\,,\qquad i,j\geqslant 0\,.
$$
The Loday product between the group subfunctors $\gm\subset K_1$ factors through the Milnor $K$-groups (see, e.g.,~\cite[\S\S\,1,\,2]{S}), that is, the product decomposes into morphisms of group functors $(\gm)^{\times m}\to K^M_m\to K_m$\,, where the second morphism sends a symbol $\{r_1,\ldots,r_m\}$ to the Loday product $r_1\cdot\ldots\cdot r_m$.}

\subsection{Weakly stable rings}

Recall the following definition given by Morrow in~\cite[Def.\,3.1]{Mor} (this is a slightly different form, which is equivalent to the one from op.\,cit.).

\begin{defin}\label{defin:wstable}
Given a natural number $k\geqslant 2$, a ring $R$ is called {\it weakly $k$-fold stable} if for any collection of elements $r_1,\ldots,r_{k-1}\in R$, there is an invertible element $r\in R^*$ such that the elements $r_1+r,\ldots,r_{k-1}+r$ are invertible in $R$.
\end{defin}

\begin{examp}\label{exam:wstable}
\hspace{0cm}
\begin{itemize}
\item[(i)]
A ring $R$ is weakly $2$-fold stable if and only if any element from $R$ is a sum of two invertible elements.
\item[(ii)]
A semi-local ring is weakly $k$-fold stable if and only if each of its residue fields contains at least $k+1$ elements,~\cite[Rem.\,3.3]{Mor}.
\item[(iii)]
For any ring $A$, the ring of Laurent series $A((t))=A[[t]][t^{-1}]$ is weakly $k$-fold stable for all $k\geqslant 2$. In fact, one can take an invertible element $r$ in  Definition~\ref{defin:wstable} to be equal to an element~$t^i$ for a suitable~\mbox{$i\in \Z$}.
\end{itemize}
\end{examp}

\begin{rmk}\label{rmk:notstable}
\hspace{0cm}
\begin{itemize}
\item[(i)]
Let $k \geqslant 1$.
Recall from~\cite{vdK2} that a ring $R$ is {\it $k$-fold stable} if for any collection of elements $r_1,s_1,\ldots,r_k,s_k$ with $r_i,s_i\in R$ such that
$$
r_1R+s_1R=\ldots=r_kR+s_kR=R\,,
$$
there is an element $r\in R$ such that $r_1+rs_1,\ldots,r_k+rs_k\in R^*$. Note that if $k \geqslant 2$, then  $k$-fold stability implies $(k-1)$-fold stability and also  implies weak $k$-fold stability,~\cite[\S\,3.1]{Mor}.
\item[(ii)]
Following the same idea as in~\cite[Rem.\,3.5]{Mor}, we observe that the ring of Laurent series $A((t))$ can be even not $1$-fold stable and hence not $k$-fold stable for any $k\geqslant 1$. Namely, suppose that $\Spec(A)$ is connected, $A$ has no nilpotent elements and has a non-invertible element $a\in A$. Then the pair $a,at^{-1}+1$ breaks the $1$-fold stability for $A((t))$, that is, there is no a Laurent series $f\in A((t))$ such that $a+f(at^{-1}+1)\in A((t))^*$. Indeed, one shows that the first non-zero coefficient in $a+f(at^{-1}+1)$ belongs to the ideal $(a)\subset A$. Hence this coefficient is not invertible in $A$. Recall that for $A$ as above, a Laurent series in $A((t))$ is invertible if and only if its first non-zero coefficient is invertible, see~\cite[Lem.\,1.3]{CC1}
and~\cite[Lem.\,0.8]{CC2}. This implies that $a+f(at^{-1}+1)$ is not invertible.
\end{itemize}
\end{rmk}

\subsection{Tangent spaces to group functors}\label{subsection:statementtangent}

Below $\varepsilon$ denotes a formal variable that satisfies $\varepsilon^2=0$. Thus for any ring $R$, we have an isomorphism $R[\varepsilon]\simeq R[x]/(x^2)$, where $x$ is a formal variable.

\begin{defin}\label{defin:tang}
Given a group functor $F$, a {\it tangent space} $TF$ to $F$ is the group functor
$$
TF(R):=\Ker\big(F\big(R[\varepsilon]\big)\to F(R)\big)\,.
$$
\end{defin}

In particular, there is a decomposition $F\big(R[\varepsilon]\big)\simeq F(R)\times TF(R)$.

\begin{examp}\label{examp:tangspace}
\hspace{0cm}
\begin{itemize}
\item[(i)]
We have that $d(\varepsilon^2)=2\varepsilon d\varepsilon=0$ in $\Omega^1_{R[\varepsilon]}$ and a calculation shows that there is an isomorphism of $R$-modules
\begin{equation}\label{eq:decomforms}
T\Omega^{n+1}(R)\simeq \big(\varepsilon\,\Omega^{n+1}_R\big)\oplus \big(d\varepsilon\wedge\Omega^n_R\big)\oplus \left(\big(\varepsilon d\varepsilon\wedge\Omega^n_R\big)/2\big(\varepsilon d\varepsilon\wedge\Omega^n_R\big) \right)\,.
\end{equation}
In particular, if $\frac{1}{2}\in R$, then the last summand equals zero.
\item[(ii)]
There is a group decomposition  $R[\varepsilon]^*\simeq R^*\times (1+R\,\varepsilon)$. It follows that the subgroup $TK^M_{n+1}(R)\subset K^M_{n+1}(R)$ is generated by symbols $\{u_1,\ldots,u_n\}$, where each element $u_i\in R[\varepsilon]^*$ is either from the subgroup $R^*$ or from the subgroup $1+R\,\varepsilon$, and at least one $u_i$ is from $1+R\,\varepsilon$.
\quash{of type
$$
\{s_1,\ldots,1+r\,\varepsilon, \ldots, s_{n+1}\}\,,
$$
where $s_i\in R[\varepsilon]^*$ and $r\in R$.}
\end{itemize}
\end{examp}

Following a construction of Bloch~\cite{Blo}, we give the next definition.

\begin{defin}\label{defin:B}
Denote by
$$
B\;:\;TK^M_{n+1}\lrto \Omega^{n}
$$
the morphism of group functors obtained as the composition of $d\log\colon TK^M_{n+1}\to T\Omega^{n+1}$ and the projection to the direct summand $\Omega^{n}\simeq d\varepsilon\wedge\Omega^n$ in decomposition~\eqref{eq:decomforms}.
\end{defin}

\begin{examp}\label{rem:explB}
For any collection of invertible elements $r_1,\ldots,r_{n}\in R^*$ and any element $s\in R$, there is an equality
$$
B\,\{1+sr_1\ldots r_n\,\varepsilon,r_1,\ldots,r_n\}=sdr_1\wedge\ldots \wedge dr_n\,.
$$
Indeed, we have that
$$
\frac{d(1+sr_1\ldots r_n\,\varepsilon)}{1+sr_1\ldots r_n\,\varepsilon}=(1-sr_1\ldots r_n\,\varepsilon)d(sr_1\ldots r_n\,\varepsilon)=
$$
$$
=\varepsilon d(sr_1\ldots r_n)+sr_1\ldots r_n d\varepsilon-(sr_1\ldots r_n)^2\varepsilon d\varepsilon\,.
$$
Therefore, there are equalities
\begin{multline*}
d\log\,\{1+sr_1\ldots r_n\,\varepsilon,r_1,\ldots,r_n\}=\frac{d(1+sr_1\ldots r_n\,\varepsilon)}{1+sr_1\ldots r_n\,\varepsilon}\wedge\frac{dr_1}{r_1}\wedge\ldots\wedge\frac{dr_n}{r_n}= \\
=\varepsilon \left(d(sr_1\ldots r_n)\wedge\frac{dr_1}{r_1}\wedge\ldots\wedge\frac{dr_n}{r_n}\right)+d\varepsilon\wedge \left(sdr_1\wedge\ldots\wedge dr_n\right)-
\\
-\varepsilon d\varepsilon\wedge \left(s^2r_1\ldots r_n dr_1\wedge\ldots\wedge dr_n\right)\,.
\end{multline*}
The projection of this element to the direct summand $\Omega_R^{n}\simeq d\varepsilon\wedge\Omega_R^n$ in decomposition~\eqref{eq:decomforms} equals $sdr_1\wedge\ldots\wedge dr_n$.
\end{examp}

\medskip

Here is the main result of the paper.

\begin{theor}\label{thm:tangentMilnor}
Let $R$ be a ring such that $R$ contains $\frac{1}{2}$ and is weakly $5$-fold stable. Then the homomorphism $B\colon TK^M_{n+1}(R)\to\Omega_R^{n}$ is an isomorphism.
\end{theor}

We do not know whether Theorem~\ref{thm:tangentMilnor} is true for weakly $4$-fold stable rings containing~$\frac{1}{2}$.

\section{Proof of the main result}

As above, $n\geqslant 0$ is a natural number and $\varepsilon$ denotes a formal variable that satisfies the relation $\varepsilon^2=0$.

\quash{In addition, denote by
$$
\varphi\;:\; K^M_{n+1}\big(R[\varepsilon]\big)\lrto K^M_{n+1}\big(R[\varepsilon]\big)
$$
the group endomorphism induced by the ring endomorphism of $R[\varepsilon]$ that sends $\varepsilon$ to $2\varepsilon$.}

\subsection{Auxiliary results}\label{subsection:auxK2}

We start with the following elementary lemma.

\begin{lemma}\label{lemma:filtr}
Let $G$ be an Abelian group. Suppose there exists an automorphism \mbox{$\varphi\colon G\to G$} such that $G$ is generated by elements $g\in G$ that satisfy $\varphi(g)=2^i\cdot g$ for some natural number $i\geqslant 1$ depending on $g$. Then the group $G$ is uniquely $2$-divisible.
\end{lemma}
\begin{proof}
We need to show that multiplication by $2$ is a bijection from $G$ to itself. Define the following increasing filtration on $G$: put $F_0\,G=\{0\}$ and let $F_l\,G$, $l\geqslant 1$, be the subgroup generated by elements $g\in G$ that satisfy $\varphi(g)=2^i\cdot g$ for some natural number~$i$, ${1\leqslant i\leqslant l}$. It is enough to show that multiplication by $2$ is a bijection on each adjoint quotient $F_l\,G/F_{l-1}\,G$, $l\geqslant 1$. It is easily seen that the filtration is preserved by any endomorphism of the group $G$ that commutes with $\varphi$. In particular, the filtration is preserved by the automorphism $\varphi$ itself and by the inverse $\varphi^{-1}$. Therefore the automorphism $\varphi$ induces an automorphism on each adjoint quotient $F_l\,G/F_{l-1}\,G$. On the other hand, by construction of the filtration, the automorphism $\varphi$ acts on $F_l\,G/F_{l-1}\,G$ as multiplication by~$2^l$. This finishes the proof.
\end{proof}

Lemma~\ref{lemma:filtr} implies that following useful result.

\begin{prop}\label{prop:divis}
Assume that a ring $R$ contains $\frac{1}{2}$. Then the group $TK_{n+1}^M(R)$ is uniquely $2$-divisible.
\end{prop}
\begin{proof}
Consider a ring automorphism of $R[\varepsilon]$ that sends $\varepsilon$ to $2\varepsilon$ and is identity on $R$. It induces an automorphism $\varphi\colon K^M_{n+1}\big(R[\varepsilon]\big)\to K^M_{n+1}\big(R[\varepsilon]\big)$. Because of the equality \mbox{$1+2r\varepsilon=(1+r\varepsilon)^2$}, $r\in R$, Example~\ref{examp:tangspace}(ii) implies that $\varphi$ acts as multiplication by positive powers of~$2$ on symbols in $TK_{n+1}^M(R)$. By Lemma~\ref{lemma:filtr}, this proves the proposition.
\end{proof}

Now we prove two lemmas on the Milnor $K$-group $K_2^M\big(R[\varepsilon]\big)$.

\begin{lemma}\label{lemma:epseps}
Let $R$ be any ring.
\begin{itemize}
\item[(i)]
For all elements $a\in R^*$ and $b\in R$ such that $1-a\in R^*$, there is an equality in~$K^M_2\big(R[\varepsilon]\big)$
$$
2\left\{1+\frac{b}{a}\,\varepsilon,1+\frac{b}{1-a}\,\varepsilon\right\}=0\,.
$$
\item[(ii)]
For all elements $r_1,r_2\in R^*$ such that $r_1+r_2\in R^*$, there is an equality in $K_2^M\big(R[\varepsilon]\big)$
$$
2\{1+r_1\varepsilon,1+r_2\,\varepsilon\}=0\,.
$$
\item[(iii)]
Suppose that $\frac{1}{2}\in R$ and $R$ is weakly $4$-fold stable. Then for all elements $r_1,r_2\in R$, there is an equality in $K_2^M\big(R[\varepsilon]\big)$
$$
\{1+r_1\,\varepsilon,1+r_2\,\varepsilon\}=0\,.
$$
\end{itemize}
\end{lemma}
\begin{proof}
$(i)$ We have the Steinberg relation in $K^M_2\big(R[\varepsilon]\big)$
\begin{equation}\label{eq:0miln}
\{a+b\,\varepsilon,1-a-b\,\varepsilon\}=0\,.
\end{equation}
Note that
\begin{equation}\label{eq:trivial}
a+b\,\varepsilon=a\left(1+\frac{b}{a}\,\varepsilon\right)\,,\qquad
1-a-b\,\varepsilon=(1-a)\left(1-\frac{b}{1-a}\,\varepsilon\right)\,.
\end{equation}
Using multilinearity and the Steinberg relation $\{a,1-a\}=0$ in $K^M_2(R)$, we see that~\eqref{eq:0miln} and~\eqref{eq:trivial} imply the equality
\begin{equation}\label{eq:1miln}
\left\{a,1-\frac{b}{1-a}\,\varepsilon\right\}+\left\{1+\frac{b}{a}\,\varepsilon,1-a\right\}+
\left\{1+\frac{b}{a}\,\varepsilon,1-\frac{b}{1-a}\,\varepsilon\right\}=0\,.
\end{equation}
Applying the automorphism of $R[\varepsilon]$ that sends $\varepsilon$ to $-\varepsilon$ and is identity on $R$, we get the equality
\begin{equation}\label{eq:2miln}
\left\{a,1+\frac{b}{1-a}\,\varepsilon\right\}+\left\{1-\frac{b}{a}\,\varepsilon,1-a\right\}+
\left\{1-\frac{b}{a}\,\varepsilon,1+\frac{b}{1-a}\,\varepsilon\right\}=0\,.
\end{equation}
Since $(1+r\varepsilon)^{-1}=1-r\varepsilon$ for any $r\in R$, the sum of~\eqref{eq:1miln} and~\eqref{eq:2miln} gives
$$
2\left\{1+\frac{b}{a}\,\varepsilon,1-\frac{b}{1-a}\,\varepsilon\right\}=0\,.
$$
Taking the inverse element in the group $K_2\big(R[\varepsilon]\big)$, we obtain item~$(i)$.

$(ii)$ Apply item~$(i)$ with
$$
a=\frac{r_2}{r_1+r_2}\,,\qquad b=\frac{r_1r_2}{r_1+r_2}\,.
$$

$(iii)$ Since $(1+s_2\,\varepsilon)\cdot(1+s_2\,\varepsilon)=1+(s_1+s_2)\,\varepsilon$ for all $s_1,s_2\in R$, Example~\ref{exam:wstable}(i) implies that we may assume $r_1,r_2\in R^*$. Since $R$ is weakly $4$-fold stable, there exists $r\in R^*$ such that
$$
r_1+r,\,r_2+r,\,\frac{r_1+r_2}{2}+r\in R^*\,.
$$
In particular, we have that $r_1+r_2+2r\in R^*$ and $2r\in R^*$. It follows from item~$(ii)$ that there are equalities
$$
0=2\{1+(r_1+r)\varepsilon,1+(r_2+r)\varepsilon\}=2\{1+r_1\,\varepsilon,1+r_2\,\varepsilon\}+
2\{1+r\,\varepsilon,1+r_2\,\varepsilon\}+
$$
$$
+2\{1+r_1\,\varepsilon,1+r\,\varepsilon\}+2\{1+r\,\varepsilon,1+r\,\varepsilon\}=
2\{1+r_1\,\varepsilon,1+r_2\,\varepsilon\}\,.
$$
Together with Proposition~\ref{prop:divis}, this proves item~$(iii)$.
\end{proof}

\begin{lemma}\label{lemma:cool}
Let $R$ be a ring such that $\frac{1}{2}\in R$ and $R$ is weakly $4$-fold stable. Let $r_1,\ldots,r_N\in R^*$, $N\geqslant 2$, be such that $r_1+\ldots+r_N=0$. Then there is an equality in~$K^M_2\big(R[\varepsilon]\big)$
$$
\{1+r_1\,\varepsilon,r_1\}+\ldots+\{1+r_N\,\varepsilon,r_N\}=0\,.
$$
\end{lemma}
\begin{proof}
We use induction on $N$. To prove the lemma for the case $N=2$, observe that
$$
\{1+r_1\,\varepsilon,r_1\}+\{1+r_2\,\varepsilon,r_2\}=
\{1+r_1\,\varepsilon,r_1\}+\{1-r_1\,\varepsilon,-r_1\}=
$$
$$
=\{1+r_1\,\varepsilon,r_1\}-\{1+r_1\,\varepsilon,-r_1\}=
\{1+r_1\,\varepsilon,-1\}
$$
and use Proposition~\ref{prop:divis}.

To prove the lemma for the case $N=3$, by the case $N=2$, it is enough to show the equality
$$
\{1+(r_1+r_2)\,\varepsilon,r_1+r_2\}-\{1+r_1\,\varepsilon,r_1\}-\{1+r_2\,\varepsilon,r_2\}=0\,.
$$
This is equivalent to the equality
\begin{equation}\label{eq:SY}
\{1+(r_1+r_2)\,\varepsilon,r_1+r_2-r_1r_2\,\varepsilon\}-
\{1+r_1\,\varepsilon,r_1\}-\{1+r_2\,\varepsilon,r_2\}=0\,,
\end{equation}
because the symbol $\big\{1+(r_1+r_2)\,\varepsilon,1-\frac{r_1r_2}{r_1+r_2}\,\varepsilon\big\}$ equals zero by Lemma~\ref{lemma:epseps}$(iii)$. Formula~\eqref{eq:SY} is essentially proved by Suslin and Yarosh in~\cite[Lem.\,3.5]{SY} (see also~\cite[Sublem.\,3.3]{Ker}). Namely, by multilinearity, the left hand side in~\eqref{eq:SY} is equal to
\begin{equation}\label{eq:SY2}
\left\{1+r_1\,\varepsilon,1+\frac{r_2}{r_1}-r_2\,\varepsilon\right\}+
\left\{1+r_2\,\varepsilon,1+\frac{r_1}{r_2}-r_1\,\varepsilon\right\}\,.
\end{equation}
Applying the Steinberg relation twice, we obtain the equalities
$$
\left\{1+r_1\,\varepsilon,1+\frac{r_2}{r_1}-r_2\,\varepsilon\right\}=
\left\{(1+r_1\,\varepsilon)\left(-\frac{r_2}{r_1}+r_2\,\varepsilon\right),1
+\frac{r_2}{r_1}-r_2\,\varepsilon\right\}=
$$
$$
=\left\{-\frac{r_2}{r_1},1+\frac{r_2}{r_1}-r_2\,\varepsilon\right\}=
\left\{-\frac{r_2}{r_1},\left(1+\frac{r_2}{r_1}\right)
\left(1-\frac{r_1r_2}{r_1+r_2}\,\varepsilon\right)\right\}=
$$
$$
=\left\{-\frac{r_2}{r_1},1-\frac{r_1r_2}{r_1+r_2}\,\varepsilon\right\}\,.
$$
Since the last expression is antisymmetric with respect to $r_1$ and $r_2$, we obtain that~\eqref{eq:SY2} equals zero. This  proves the case $N=3$.

Now let us make the induction step for $N\geqslant 4$. For short, put $\langle s\rangle:=\{1+s\,\varepsilon,s\}$, where $s\in R^*$. The case $N=2$ asserts that $\langle-s\rangle=-\langle s\rangle$ for any $s\in R^*$. The case $N=3$ asserts that $\langle s_1+s_2\rangle=\langle s_1\rangle+\langle s_2\rangle$ for all $s_1,s_2\in R^*$ such that $s_1+s_2\in R^*$. Since $R$ is weakly $4$-fold stable, there is $r\in R^*$ such that
$$
r+r_1,\,r+r_1+r_2,\,r+r_1+r_2+r_3\in R^*\,.
$$
We obtain the equalities
$$
\langle r_1\rangle+\ldots+\langle r_N\rangle=-\langle r\rangle+\langle r\rangle+\langle r_1\rangle+\ldots+\langle r_N\rangle=
$$
$$
=\langle -r\rangle+\langle r+r_1\rangle+\langle r_2\rangle+\ldots+\langle r_N\rangle=
\langle -r\rangle+\langle r+r_1+r_2\rangle+\langle r_3\rangle+\ldots+\langle r_N\rangle=
$$
\begin{equation}\label{SY3}
=\langle -r\rangle+\langle r+r_1+r_2+r_3\rangle+\langle r_4\rangle+\ldots+\langle r_N\rangle\,.
\end{equation}
Expression~\eqref{SY3} contains $N-1$ summands and all elements in angle parenthesis are from~$R^*$. Besides,
we have
$$
(-r)+(r+r_1+r_2+r_3)+r_4+\ldots+r_N=0\,.
$$
Therefore,
by the induction hypothesis, we obtain that~\eqref{SY3} equals zero. This  proves the lemma.
\end{proof}

Finally, recall the following result, which is proved in~\cite[Lem.\,3.6]{Mor} with a method of Nesterenko and Suslin from~\cite[Lem.\,3.2]{NS} (see also Lemma~2.2 from the paper  of Kerz~\cite{Ker}).

\begin{lemma}\label{lemma:Morrow}
Let $R$ be a weakly $5$-fold stable ring. Then for all elements $r,s\in R^*$, there are equalities in $K_2^M(R)$ $$
\{r,-r\}=0\,,\qquad \{r,s\}=-\{s,r\}\,.
$$
\end{lemma}

Here, the second equality follows in a standard way from the first one and the identity
$$
\{r,s\}+\{s,r\}=\{rs,-rs\}-\{r,-r\}-\{s,-s\}\,.
$$

\subsection{Proof of Theorem~\ref{thm:tangentMilnor}}

If $R$ is a weakly $2$-fold stable ring, then by Example~\ref{exam:wstable}(i), the group $\Omega^n_R$ is generated by differential forms~${sdr_1\wedge\ldots\wedge dr_n}$, where \mbox{$r_1,\ldots,r_n\in R^*$}, $s\in R$. The next lemma is similar to~\cite[Lem.\,1.8]{Blo} and the proof is straightforward.

\begin{lemma}\label{lemma:dull}
Let $R$ be a weakly $2$-fold stable ring. Let $M$ be an $R$-module and let
$f\colon (R^*)^{\times n}\to M$ be a map of sets such that
\begin{itemize}
\item[(i)]
$f$ is alternating: for all $r_1,\ldots,r_n\in R^*$ and $1\leqslant i\leqslant n-1$, we have
$$
f(r_1,\ldots,r_i,r_{i+1},\ldots,r_n)=-f(r_1,\ldots,r_{i+1},r_i,\ldots,r_n)
$$
and if $r_i=r_{i+1}$, then ${f(r_1,\ldots,r_i,r_{i+1},\ldots,r_n)=0}$;
\item[(ii)]
$f$ satisfies the Leibniz rule: for all $r_1,r_1',r_2,\ldots,r_n\in R^*$, we have
$$
f(r_1r_1',r_2,\ldots,r_n)=r_1f(r_1',r_2,\ldots,r_n)+r'_1f(r_1,r_2,\ldots,r_n)\,;
$$
\item[(iii)]
$f$ is additively multilinear: for all $r_{11},\ldots,r_{1N},r_2\ldots,r_n\in R^*$ (where $N\geqslant 2$) such that $r_{11}+\ldots+r_{1N}=0$, we have
$$
f(r_{11},r_2,\ldots,r_n)+\ldots+f(r_{1N},r_2,\ldots,r_n)=0\,.
$$
\end{itemize}
Then there is a unique homomorphism of $R$-modules
$F\colon \Omega^n_R\to M$
such that for all \mbox{$r_1,\ldots,r_n\in R^*$} and $s\in R$, we have that
$F(sdr_1\wedge\ldots \wedge dr_n)=sf(r_1,\ldots,r_n)$.
\end{lemma}

Note that it is important to require in Lemma~\ref{lemma:dull}$(iii)$ that $N\geqslant 2$ is an arbitrary natural number, because the sum of invertible elements of a ring is not necessarily an invertible element. However for a weakly $4$-fold stable ring $R$, the case of an arbitrary number $N$ can be reduced to the cases $N=2,3$, which can be shown by the method from the end of the proof of Lemma~\ref{lemma:cool}.

\medskip

In order to apply Lemma~\ref{lemma:dull}, we define an $R$-module structure on the group~$TK^M_{n+1}(R)$ (cf.~\cite[Lem.\,1.4,\,1.5]{Blo}).

\begin{prop}\label{corol:milnor}
Let $R$ be a weakly $5$-fold stable ring such that $\frac{1}{2}\in R$. Then the following is true:
\begin{itemize}
\item[(i)]
The group $TK_{n+1}^M(R)$ is generated by symbols of type $\{1+sr_1\ldots r_n\,\varepsilon,r_1,\ldots,r_n\}$, where  \mbox{$r_1,\ldots,r_n\in R^*$}, $s\in R$.
\item[(ii)]
The action of $R$ on the ring $R[\varepsilon]$ by endomorphisms that send $\varepsilon$ to $a\,\varepsilon$, where $a\in R$, and are identity on $R$ defines an $R$-module structure on $TK^M_{n+1}(R)$ such that \begin{equation}\label{eq:action}
a\;:\;\{1+sr_1\ldots r_n\,\varepsilon,r_1,\ldots,r_n\}\longmapsto
\{1+asr_1\ldots r_n\,\varepsilon,r_1,\ldots,r_n\}\,.
\end{equation}
\end{itemize}
\end{prop}
\begin{proof}
$(i)$ It is easy to see that if $R$ is weakly $k$-fold stable, then so is the ring~$R[\varepsilon]$. Hence Lemma~\ref{lemma:Morrow} holds for the weakly $5$-fold stable ring~$R[\varepsilon]$. Now combine this with Example~\ref{examp:tangspace}(ii) and Lemma~\ref{lemma:epseps}$(iii)$.

$(ii)$ The only non-trivial statement is distributivity with respect to elements from $R$, that is, the equality $(a+b)v=av+bv$ for all $a,b\in R$, $v\in TK^M_{n+1}(R)$. This follows from item~$(i)$ and formula~\eqref{eq:action}.
\end{proof}

Formula~\eqref{eq:action} implies that the map $B$ (see Definition~\ref{defin:B}) is a homomorphism of \mbox{$R$-modules}.

\medskip

Now we are ready to prove the main result of the paper.

\begin{proof}[Proof of Theorem~\ref{thm:tangentMilnor}]
Following the strategy from~\cite{Blo}, we construct the inverse to the map $B$. By Proposition~\ref{corol:milnor}$(ii)$, the group $TK_{n+1}^M(R)$ is an $R$-module. Let us apply Lemma~\ref{lemma:dull} to the map
$$
f\;:\;(R^*)^{\times n}\longrightarrow TK_{n+1}^M(R)\,,\qquad (r_1,\ldots,r_n)\longmapsto \{1+r_1\ldots r_n\,\varepsilon,r_1,\ldots,r_n\}\,.
$$
For this, we need to check that $f$ satisfies all conditions from Lemma~\ref{lemma:dull}. We use notation of this lemma. The map $f$ is alternating by Proposition~\ref{prop:divis} and Lemma~\ref{lemma:Morrow}. The Leibniz rule for $f$ follows from the equalities (see also formula~\eqref{eq:action})
$$
\{1+r_1r'_1r_2\ldots r_n\,\varepsilon,r_1r'_1,r_2,\ldots,r_n\}=
$$
$$
=\{1+r_1r'_1r_2\ldots r_n\,\varepsilon,r'_1,r_2,\ldots,r_n\}+
\{1+r_1r'_1r_2\ldots r_n\,\varepsilon,r_1,r_2,\ldots,r_n\}=
$$
$$
=r_1\{1+r'_1r_2\ldots r_n\,\varepsilon,r'_1,r_2,\ldots,r_n\}+r'_1\{1+r_1r_2\ldots r_n\,\varepsilon,r_1,r_2,\ldots,r_n\}\,.
$$
Finally, we show that $f$ is an additively multilinear map. Indeed, by Lemma~\ref{lemma:cool}, there is an equality in $K_2^M\big(R[\varepsilon]\big)$
$$
\{1+r_{11}\,\varepsilon,r_{11}\}+\ldots+\{1+r_{1N}\,\varepsilon,r_{1N}\}=0\,.
$$
Applying the action of the element $r_2\ldots r_n\in R$ (see formula~\eqref{eq:action}), we obtain the equality
$$
\{1+r_{11}r_2\ldots r_n\,\varepsilon,r_{11}\}+\ldots+\{1+r_{1N}r_2\ldots r_n\,\varepsilon,r_{1N}\}=0\,.
$$
Taking the product with the symbol $\{r_2,\ldots, r_n\}\in K^M_{n-1}(R)$, we get the required equality.

Thus by Lemma~\ref{lemma:dull}, we obtain a homomorphism of $R$-modules $F\colon \Omega^n_R\to TK^M_{n+1}(R)$ such that for all \mbox{$r_1,\ldots,r_n\in R^*$} and $s\in R$, we have that
\begin{equation}\label{eq:fbar}
F(sdr_1\wedge\ldots\wedge dr_n)=\{1+sr_1\ldots r_n\,\varepsilon,r_1,\ldots,r_n\}\,.
\end{equation}
By Proposition~\ref{corol:milnor}$(i)$, $TK^M_{n+1}(R)$ is generated by symbols from the right hand side of~\eqref{eq:fbar}. Also,~$\Omega^n_R$ is generated  by differential forms from the left hand side of~\eqref{eq:fbar}. Therefore formula~\eqref{eq:fbar} together with Example~\ref{rem:explB} imply that the maps $B$ and $F$ are inverse to each other. This finishes the proof of Theorem~\ref{thm:tangentMilnor}.
\end{proof}

\quash{\begin{rmk}
Actually, the above reasonings show that $B\colon K_2^M\big(R[\varepsilon]\big)\to \Omega^1_R$ is an isomorphism for any $R$ such that $\frac{1}{2}\in R$ and $R$ is weakly $4$-fold stable (not necessarily weakly $5$-fold stable).
\end{rmk}}

S.\,O.~Gorchinskiy

Steklov Mathematical Institute of RAS

{\it E-mail}: gorchins@mi.ras.ru

\medskip

D.\,V.~Osipov

Steklov Mathematical Institute of RAS

{\it E-mail}: d\_osipov@mi.ras.ru

\end{document}